\documentclass[a4paper,12pt]{article}
\usepackage{tikz}
\usetikzlibrary{matrix,arrows,decorations.markings}
\usepackage{amsmath,amsfonts,amssymb,amsthm,url,textcomp}
\usepackage{csquotes}
\usepackage{a4wide}
\usepackage[numbers]{natbib}
\usepackage{fancyhdr}
\usepackage{anyfontsize}
\usepackage{hyperref}
\usepackage{hhline}
\usepackage{leftidx}
\usepackage{mdframed}
\usepackage{enumitem}

\allowdisplaybreaks

\newtheorem{theorem}{Theorem}\numberwithin{theorem}{section}

\newtheorem{proposition}[theorem]{Proposition}

\newtheorem{question}[theorem]{Question}
\newtheorem{problem}[theorem]{Problem}
\newtheorem{theoremm}{Theorem}\numberwithin{theoremm}{subsection}
\newtheorem{lemmma}[theoremm]{Lemma}

\newtheorem{nottation}[theoremm]{Notation}
\newtheorem{propposition}[theoremm]{Proposition}

\numberwithin{theoremmm}{subsubsection}

\theoremstyle{remark}

\theoremstyle{definition}
\newtheorem{deffinition}[theoremm]{Definition}

\newcommand{\Rad}{\operatorname{Rad}}
\newcommand{\Aut}{\operatorname{Aut}}
\newcommand{\Alt}{\operatorname{Alt}}
\newcommand{\PSL}{\operatorname{PSL}}

\newcommand{\lcm}{\operatorname{lcm}}

\newcommand{\ord}{\operatorname{ord}}
\newcommand{\Sym}{\operatorname{Sym}}
\newcommand{\Hol}{\operatorname{Hol}}
\newcommand{\A}{\operatorname{A}}

\newcommand{\C}{\operatorname{C}}

\newcommand{\mao}{\operatorname{mao}}
\newcommand{\maffo}{\operatorname{maffo}}

\newcommand{\Inn}{\operatorname{Inn}}
\newcommand{\id}{\operatorname{id}}

\newcommand{\fix}{\operatorname{fix}}
\newcommand{\Out}{\operatorname{Out}}
\newcommand{\e}{\mathrm{e}}

\newcommand{\PGL}{\operatorname{PGL}}

\newcommand{\GL}{\operatorname{GL}}

\newcommand{\Comp}{\operatorname{Comp}}

\newcommand{\Mod}[1]{\ (\textup{mod}\ #1)}

\newcommand{\conj}{\operatorname{conj}}
\newcommand{\IN}{\mathbb{N}}

\newcommand{\IF}{\mathbb{F}}

\newcommand{\End}{\operatorname{End}}

\newcommand{\IZ}{\mathbb{Z}}

\newcommand{\Hom}{\operatorname{Hom}}

\newcommand{\Exp}{\operatorname{Exp}}

\newcommand{\ffrak}{\mathfrak{f}}

\newcommand{\Ffrak}{\mathfrak{F}}
\newcommand{\Ord}{\operatorname{Ord}}

\begin{document}

\title{Finite groups with an affine map of large order}

\author{Alexander Bors\thanks{School of Mathematics and Statistics, Carleton University, 1125 Colonel By Drive, Ottawa ON K1S 5B6, Canada. E-mail: \href{mailto:alexanderbors@cunet.carleton.ca}{alexanderbors@cunet.carleton.ca} \newline The author was supported by the Austrian Science Fund (FWF), project J4072-N32 \enquote{Affine maps on finite groups}. \newline 2020 \emph{Mathematics Subject Classification}: Primary: 20F14. Secondary: 20D05, 20D25, 20D45. \newline \emph{Key words and phrases:} Finite group, Holomorph, Group element order, Derived length, Solvable radical.}}

\date{\today}

\maketitle

\abstract{Let $G$ be a group. A function $G\rightarrow G$ of the form $x\mapsto x^{\alpha}g$ for a fixed automorphism $\alpha$ of $G$ and a fixed $g\in G$ is called an \emph{affine map of $G$}. In this paper, we study finite groups $G$ with an affine map of large order. More precisely, we show that if $G$ admits an affine map of order larger than $\frac{1}{2}|G|$, then $G$ is solvable of derived length at most $3$. We also show that more generally, for each $\rho\in\left(0,1\right]$, if $G$ admits an affine map of order at least $\rho|G|$, then the largest solvable normal subgroup of $G$ has derived length at most $4\lfloor\log_2(\rho^{-1})\rfloor+3$.}

\section{Introduction}\label{sec1}

By classical ideas dating back to the beginning of modern abstract group theory, we have the following two important (right) group actions on each group $G$:
\begin{enumerate}
\item the regular action of $G$ on itself, defined via $g^h:=g\cdot h$ for $g,h\in G$, which allows us to view $G$ as a (regular) permutation group;
\item the natural action of the automorphism group $\Aut(G)$ on $G$, defined via $g^{\alpha}:=\alpha(g)$ (function evaluation) for $g\in G$ and $\alpha\in\Aut(G)$, allowing us to study the element structure of $G$ \enquote{up to abstract equivalence}.
\end{enumerate}
These actions can be combined into just one faithful action of the external semidirect product $\Aut(G)\ltimes G$, via $x^{(\alpha,g)}:=x^{\alpha}g$ for $x,g\in G$ and $\alpha\in\Aut(G)$. The image of this action, a permutation group on $G$, is called the \emph{holomorph of $G$} and denoted by $\Hol(G)$, see \cite[p.~37]{Rob96a}. In this paper, we call the elements of $\Hol(G)$ the \emph{affine maps of $G$}.

Affine maps of groups arise naturally in various contexts. For example, the finite primitive permutation groups are classified in the celebrated O'Nan-Scott theorem, \cite[Theorem in Section 2]{LPS88a}, and several infinite families of them (more precisely, the classes HA, HS and HC described in \cite[Section 3]{Pra90a}) can be viewed as subgroups of holomorphs. Affine maps also arise in applied contexts. For instance, linear congruential pseudorandom number generators (see \cite[beginning of Section 7.3]{Nie92a}) are based on the iteration of affine maps of finite cyclic groups.

This paper, while being a theoretical contribution to the study of affine maps of general finite groups, is loosely motivated by the above mentioned application of affine maps in pseudorandom number generation. Indeed, a necessary (though not sufficient) condition for an affine map $A$ of a finite group $G$ to be useful for the construction of a \enquote{good} pseudorandom number generator is that for at least one $g\in G$, the \emph{iteration orbit of $g$ under $A$}, i.e., the set $\{g^{\left(A^n\right)}\mid n\in\IZ\}$ of elements of $G$ obtainable from $g$ by iterated application of $A$, is \enquote{large} (in fact, Niederreiter in \cite[Section 7.2, p.~164]{Nie92a} mentions the stronger condition that the smallest iteration orbit length of $A$ should be large). For the purposes of this paper, we assume that \enquote{large} here means \enquote{at least $\rho|G|$} for a suitable, given $\rho\in\left(0,1\right]$ (such as $\rho=\frac{1}{2}$). In this regard, a notable earlier result on general groups is \cite[Theorem 6.2]{JS75a}, from a 1975 paper by Jonah and Schreiber, which provides a classification of the (not necessarily finite) groups $G$ with an affine map $A$ such that $G$ as a whole is an iteration orbit of $A$.

In view of this, it is natural to ask what can be said about a finite group $G$ that achieves an affine map $A$ such that the order of $A$ (the least common multiple of the cycle lengths of $A$ on $G$) is at least $\rho|G|$. By an earlier result of the author, \cite[Theorem 1.1.3(2)]{Bor17a}, we know that the index $|G:\Rad(G)|$ of the solvable radical of $G$ (the largest solvable normal subgroup of $G$) is at most $\rho^{-5.91}$ (and only at most $\rho^{-1.78}$ if $A$ is an automorphism of $G$, see \cite[Theorem 1.1.1(3)]{Bor17a}), so $G$ is \enquote{almost solvable}. This provides a partial theoretical justification for why mostly abelian groups are studied in the context of pseudorandom number generation, but it still leaves open the question how complex $\Rad(G)$ itself can be -- for example, can it be of arbitrarily large derived length? The following main result of this paper answers this question in the negative:

\begin{theorem}\label{mainTheo}
Let $G$ be a finite group, and let $\rho\in\left(0,1\right]$.
\begin{enumerate}
\item If $G$ admits an affine map of order larger than $\frac{1}{2}|G|$, then $G$ is solvable of derived length at most $3$.
\item If $G$ admits an affine map of order at least $\rho|G|$, then the derived length of $\Rad(G)$ is at most $4\cdot\lfloor\log_2(\rho^{-1})\rfloor+3$.
\end{enumerate}
\end{theorem}

In addition to \cite[Theorem 1.1.3(2)]{Bor17a}, this provides further justfication for why abelian groups are natural candidates to achieve affine maps with large cycles.

\section{Preliminaries}\label{sec2}

\subsection{Main proof idea and overview of this paper}\label{subsec2P1}

Note that if $\maffo(G)$ denotes the largest order of an affine map of the finite group $G$, then both statements of Theorem \ref{mainTheo} provide structural restrictions on $G$ under suitable lower bounds on $\maffo(G)$. In order to prove Theorem \ref{mainTheo}, we will actually prove something stronger: We will replace $\maffo(G)$ by a certain other, technical parameter, denoted by $\Ffrak(G)$ and introduced in Notation \ref{fNot}(2) below. This function was already used in the author's earlier paper \cite{Bor17a}, see \cite[Definition 5.2.3 and the paragraph thereafter]{Bor17a}. In contrast to $\maffo$, $\Ffrak$ has the nice property that $\Ffrak(G)\leq\Ffrak(N)\cdot\Ffrak(G/N)$ whenever $N$ is a characteristic subgroup of $G$, see \cite[Lemma 5.2.4]{Bor17a} and also Lemma \ref{fLem}(3) below. This allows us to use structural reduction arguments when considering lower bounds on $\Ffrak(G)$ of the form $\rho|G|$ for a constant $\rho\in\left(0,1\right]$.

Apart from the definition of $\Ffrak$, Subsection \ref{subsec2P3} contains a few basic and known results on $\Ffrak$ and on affine maps that will be used frequently in this paper. In Subsection \ref{subsec2P4}, we collect a few more known results (not related to $\Ffrak$) that will be used in the proof of Theorem \ref{mainTheo}.

Section \ref{sec3} is concerned with the proof of Theorem \ref{mainTheo}(2). This proof does not require the classification of finite simple groups and mainly consists of the verification of a statement that resembles Theorem \ref{mainTheo}(1) but is not strong enough to imply it; this statement is formulated as Proposition \ref{mainProp}. In Section \ref{sec4}, we will strengthen Proposition \ref{mainProp} and thereby prove Theorem \ref{mainTheo}(1) by considering the $\Ffrak$-values of finite nonabelian characteristically simple groups, making use of the classification of finite simple groups. Finally, in Section \ref{sec5}, we discuss some related open problems and questions for further research.

\subsection{Notation and terminology}\label{subsec2P2}

We denote by $\IN$ the set of natural numbers (including $0$), and by $\IN^+$ the set of positive integers. The restriction of a function $f$ to a set $M$ is denoted by $f_{\mid M}$. For the value of a function $f$ at an argument $x$, we will use the notations $f(x)$ and $x^f$ interchangeably, but we prefer the latter if $f$ is an element of a semigroup of functions acting naturally on a set containing $x$; this is because we assume all semigroup actions to be on the right. If $n$ is a positive integer and $p$ is a prime, we denote by $n_p$ the largest power of $p$ dividing $n$ (which may be equal to $1$) and set $n_{p'}:=\frac{n}{n_p}$.

For a finite group $G$, we denote by $\Exp(G)$ the group exponent of $G$, by $\zeta G$ the center of $G$ and by $\Phi(G)$ the Frattini subgroup of $G$. Moreover, for $n\in\IN$, the notation $G^{(n)}$ denotes the $n$-th derived subgroup of $G$ (if $n$ is small, we may also use iterated primes, such as $G'$ instead of $G^{(1)}$, $G''$ instead of $G^{(2)}$, and so on). For an element $g\in G$, the notation $\ord(g)$ denotes the element order of $g$ in $G$, and $\conj_G(g)$ denotes the inner automorphism (conjugation) by $g$ on $G$. Finally, if $\alpha$ is an automorphism of $G$ and $g\in G$, then $\A_{\alpha,g}$ denotes the affine map $x\mapsto x^{\alpha}g$ of $G$. If $A$ is a (multiplicative) group or $K$-algebra for some field $K$ and $\alpha\in A$, we denote by $\C_A(\alpha):=\{\beta\in A\mid \alpha\beta=\beta\alpha\}$ the centralizer of $\alpha$ in $A$.

The algebraic closure of a field $K$ is denoted by $\overline{K}$, and the multiplicative group of units of a ring $R$ by $R^{\ast}$. For a prime power $q$, the notation $\IF_q$ stands for the finite field with $q$ elements. If $P(X)\in R[X]$ is a univariate polynomial over the ring $R$, then $\deg{P(X)}$ denotes the degree of $P(X)$. If $R$ is a ring and $M_1,M_2$ are $R$-modules, then $\Hom_R(M_1,M_2)$ denotes the set of $R$-module homomorphims $M_1\rightarrow M_2$, and $\End_R(M):=\Hom_R(M,M)$ if $M$ is an $R$-module. If $R$ is commutative, then $\End_R(M)$ is an $R$-algebra, with $\End_R(M)^{\ast}=\Aut_R(M)$, the group of $R$-module automorphisms of $M$.

\subsection{The functions \texorpdfstring{$\Ffrak$}{F} and \texorpdfstring{$\ffrak$}{f}, and the notion of a shift}\label{subsec2P3}

We now introduce the function $\Ffrak$ mentioned in Subsection \ref{subsec2P1}.

\begin{nottation}\label{fNot}
Let $G$ be a finite group.
\begin{enumerate}
\item For an automorphism $\alpha$ of $G$, we set $\Ffrak(\alpha):=\lcm_{x\in G}{\ord(\A_{\alpha,x})}$ and $\ffrak(\alpha):=\frac{1}{|G|}\Ffrak(\alpha)$.
\item We set $\Ffrak(G):=\max_{\alpha\in\Aut(G)}{\Ffrak(\alpha)}$ and $\ffrak(G):=\frac{1}{|G|}\Ffrak(G)=\max_{\alpha\in\Aut(G)}{\ffrak(\alpha)}$.
\end{enumerate}
\end{nottation}

The following lemma lists some important known properties of $\ffrak$ and $\Ffrak$:

\begin{lemmma}\label{fLem}
Let $G$ be a finite group. Then the following hold:
\begin{enumerate}
\item $\ffrak(G)\leq1$. In particular, the maximum element order in $\Hol(G)$ cannot exceed $|G|$.
\item Let $\alpha\in\Aut(G)$, and let $N$ be an $\alpha$-invariant normal subgroup of $G$. Denote by $\tilde{\alpha}$ the automorphism of $G/N$ induced by $\alpha$. Then $\Ffrak(\tilde{\alpha})\mid\Ffrak(\alpha)\leq\Ffrak(\tilde{\alpha})\cdot\Ffrak((\alpha_{\mid N})^{\Ffrak(\tilde{\alpha})})$.
\item Let $N$ be a characteristic subgroup of $G$. Then $\ffrak(G)\leq\ffrak(N)\cdot\ffrak(G/N)\leq\min\{\ffrak(N),\ffrak(G/N)\}$.
\end{enumerate}
\end{lemmma}

\begin{proof}
For statement (1), see \cite[Theorem 5.2.3]{Bor17a}. Statement (2) follows by inspecting \cite[proof of Lemma 5.2.4]{Bor17a}. Finally, the first inequality in statement (3) follows from statement (2) and is also the statement of \cite[Lemma 5.2.4]{Bor17a}, whereas the second inequality in statement (3) follows from statement (1).
\end{proof}

The following notions are also crucial for the study of affine maps:

\begin{deffinition}\label{shiftDef}
Let $G$ be a finite group, and let $\alpha\in\Aut(G)$.
\begin{enumerate}
\item For $x\in G$, the element $x^{\left(\alpha^{\ord(\alpha)-1}\right)}x^{\left(\alpha^{\ord(\alpha)-2}\right)}\cdots x^{\left(\alpha^2\right)}x^{\alpha}x\in G$ is called the \emph{$\alpha$-shift of $x$}.
\item The function $G\rightarrow G$ mapping $x\in G$ to the $\alpha$-shift of $x$ is called the \emph{shift function of $\alpha$} and denoted by $\sigma_{\alpha}$.
\end{enumerate}
\end{deffinition}

Some useful properties of shifts were discussed in \cite[Subsection 3.1]{Bor17a}, and we list those that we will need in this paper in the following lemma:

\begin{lemmma}\label{shiftLem}
Let $G$ be a finite group, and let $\alpha\in\Aut(G)$.
\begin{enumerate}
\item For each $x\in G$, the order of $\A_{\alpha,x}$ is $\ord(\alpha)\cdot\ord(\sigma_{\alpha}(x))$.
\item For each $x\in G$, we have $\sigma_{\alpha}(x)^{\alpha}=\sigma_{\alpha}(x)^{\left(x^{-1}\right)}$. In particular, if $G$ is abelian, then $\sigma_{\alpha}$ is a group homomorphism $G\rightarrow\fix(\alpha):=\{x\in G\mid x^{\alpha}=x\}$.
\item Assume that $G$ is centerless and that $\alpha=\conj_G(r)$ is an inner automorphism of $G$. Then $\ord(\A_{\alpha,x})=\lcm(\ord(r),\ord(xr))$. In particular, $\Ffrak(\alpha)=\Exp(G)$.
\end{enumerate}
\end{lemmma}

\begin{proof}
Statement (1) follows from a power computation in $\Hol(G)$, see \cite[paragraph after Lemma 3.1.3]{Bor17a}, where affine maps were defined via left actions. For statement (2), see \cite[Lemma 3.1.3(1)]{Bor17a}, and for statement (3), see \cite[Lemma 3.1.6(1)]{Bor17a}.
\end{proof}

\subsection{A few more basic results}\label{subsec2P4}

The following simple lemma is useful when studying automorphisms of finite groups with a proper nontrivial characteristic subgroup $N$, especially when some nontrivial information is available on the restriction $\alpha_{\mid N}$ or the automorphism of $G/N$ induced by $\alpha$:

\begin{lemmma}\label{abKerLem}
Let $G$ be a finite group, let $\alpha$ be an automorphism of $G$ and let $N$ be an $\alpha$-invariant normal subgroup of $G$. Denote by $\gamma:G\rightarrow\Aut(N)$ the conjugation action of $G$ on $N$. Then for all $g\in G$, we have $(g^{\gamma})^{\alpha_{\mid N}}=(g^{\alpha})^{\gamma}$.
\end{lemmma}

\begin{proof}
Let $n\in N$. Then
\begin{align*}
n^{\left((g^{\gamma})^{\alpha_{\mid N}}\right)} &=n^{(\alpha_{\mid N})^{-1}g\alpha_{\mid N}}=(n^{(\alpha^{-1})})^{g\alpha_{\mid N}}=(g^{-1}n^{(\alpha^{-1})}g)^{\alpha_{\mid N}}=(g^{\alpha})^{-1}ng^{\alpha} \\
&=n^{(g^{\alpha})}=n^{\left((g^{\alpha})^{\gamma}\right)}.
\end{align*}
\end{proof}

Recall the Frobenius normal form (or rational canonical form) of vector space endomorphisms (see e.g.~\cite[Theorem 4.17, p.~238]{AW92a}): If $K$ is a field, $V$ is a finite-dimensional $K$-vector space and $\varphi\in\End_K(V)$, then $\varphi$ is $\Aut_K(V)$-conjugate to a block diagonal matrix whose blocks each are of the form $\Comp(P(X))$ where $P(X)=X^d+a_{d-1}X^{d-1}+\cdots+a_1X+a_0$ is a monic univariate polynomial over $K$ and $\Comp(P(X))$, the so-called \emph{companion matrix of $P(X)$}, is of the form
\[
\begin{pmatrix}0 & 0 & \cdots & 0 & -a_0 \\ 1 & 0 & \cdots & 0 & -a_1 \\ 0 & 1 & \cdots & 0 & -a_2 \\ \vdots & \vdots & \cdots & \vdots & \vdots \\ 0 & 0 & \cdots & 1 & -a_{d-1}\end{pmatrix}.
\]
Equivalently, $\Comp(P(X))$ is the matrix representation of the $K$-endomorphism
\[
R(X)+(P(X))\mapsto XR(X)+(P(X))
\]
of $K[X]/(P(X))$ with respect to the ordered basis $(X^i+(P(X)))_{i=0,1,\ldots,d-1}$. Therefore, the Chinese remainder theorem implies that if $P(X)=Q_1(X)^{e_1}\cdots Q_r(X)^{e_r}$ is the essentially unique factorization of $P(X)$ into pairwise coprime powers of monic irreducible polynomials, then $\Comp(P(X))$ is similar to a block diagonal matrix whose blocks are $\Comp(Q_i(X)^{e_i})$ for $i=1,\ldots,s$. This shows that in fact, each $K$-endomorphism $\varphi$ of $V$ has a block diagonal form each of whose blocks is the companion matrix of a power of an irreducible monic polynomial from $K[X]$. This normal form of $\varphi$, which is unique up to reordering the blocks, is called the \emph{primary rational canonical form of $\varphi$}, and its diagonal blocks are called the \emph{primary Frobenius blocks of $\varphi$}.

One of the many applications of (primary) rational canonical forms is the determination of matrix centralizers:

\begin{lemmma}\label{centLem}
Let $K$ be a field, let $V$ be a finite-dimensional $K$-vector space, let $\varphi\in\End_K(V)$ with primary Frobenius blocks $\Comp(P_i(X)^{e_i})$ for $i=1,\ldots,s$, corresponding to a direct decomposition $V=U_1\oplus U_2\oplus\cdots\oplus U_s$. Let $X$ be a variable, and view $V$ and each $U_i$ as a $K[X]$-module by defining $X\cdot v:=v^{\varphi}$ for $v\in V$. Let $\psi\in\End_K(V)$. The following are equivalent:
\begin{enumerate}
\item $\psi\in\C_{\End_K(V)}(\varphi)$.
\item There is an $(s\times s)$-matrix $(\pi_{i,j})_{i,j=1,\ldots,s}$ such that for all $i,j\in\{1,\ldots,s\}$, we have $\pi_{i,j}\in\Hom_{K[X]}(U_j,U_i)$, and such that for each $v\in V$, writing $v=u_1+\cdots+u_s$ with $u_i\in U_i$ for $i=1,\ldots,s$, we have $v^{\psi}=\sum_{i,j=1}^s{\pi_{i,j}(u_j)}$.
\end{enumerate}
In particular, if the polynomials $P_i(X)$ for $i=1,\ldots,s$ are pairwise distinct, then
\begin{align*}
\C_{\Aut_K(V)}(\varphi) &=\C_{\Aut_K(U_1)}(\varphi_{\mid U_1})\times\cdots\times\C_{\Aut_K(U_s)}(\varphi_{\mid U_s}) \\
&=\Aut_{K[X]}(U_1)\times\cdots\times\Aut_{K[X]}(U_s) \\
&\cong (K[X]/(P_1(X)^{e_1}))^{\ast}\times\cdots\times(K[X]/(P_s(X)^{e_s}))^{\ast}.
\end{align*}
\end{lemmma}

\begin{proof}
The main statement is a special case of \cite[Proposition 3.3.15]{AW92a}, noting that $\C_{\End_K(W)}(\varphi)=\End_{K[X]}(W)$ for $W\in\{V,U_1,\ldots,U_s\}$ under the specified $K[X]$-module structures. For the \enquote{In particular}, note that by \cite[Lemma 5.5.2]{AW92a}, if the $P_i(X)$ are pairwise distinct, then $\Hom_{K[X]}(U_j,U_i)=\{0\}$ for $i\not=j$, whence
\[
\C_{\End_K(V)}(\varphi)=\C_{\End_K(U_1)}(\varphi_{\mid U_1})\times\cdots\times\C_{\End_K(U_s)}(\varphi_{\mid U_s}),
\]
and an element of the right-hand side is an automorphism of $K$ if and only if each of its restrictions to one of the subspaces $U_i$ is an automorphism of that subspace. Finally, that $\Aut_{K[X]}(U_i)$ is isomorphic to $(K[X]/(P_i(X)^{e_i}))^{\ast}$ follows from \cite[Lemma 5.5.2]{AW92a} and the $K[X]$-module isomorphism $U_i\cong K[X]/(P_i(X)^{e_i})$ noted in the text passage after the proof of Lemma \ref{abKerLem} above.
\end{proof}

The order of an invertible matrix $M$ over a finite field can also be read off from its primary rational canonical form:

\begin{propposition}\label{frobOrdProp}
Let $K$ be a finite field of characteristic $p$.
\begin{enumerate}
\item Let $P(X)\in K[X]$ be monic and irreducible such that $P(0)\not=0$, and let $k\in\IN^+$. Let $\alpha\in\overline{K}$ be any root of $P(X)$, and denote by $\ord(\alpha)$ the order of $\alpha$ in $\overline{K}^{\ast}$. Then $\ord(\Comp(P(X)^e))=\ord(\alpha)\cdot p^{\lceil\log_p(e)\rceil}$.
\item If $M$ is an invertible square matrix over $K$ and $\Comp(P_i(X)^{e_i})$ for $i=1,\ldots,s$ are the primary Frobenius blocks of $M$, then $P_i(0)\not=0$ for all $i=1,\ldots,s$ and $\ord(M)=\lcm\{\ord(\Comp(P_i(X)^{e_i}))\mid i=1,\ldots,s\}$.
\end{enumerate}
\end{propposition}

\begin{proof}
This follows from \cite[Theorem 3.8]{LN97a} and \cite[Theorems 4 and 5]{Her05a}.
\end{proof}

We conclude this subsection with the following classification of automorphisms $\alpha$ of finite elementary abelian groups $G$ such that $\ord(\alpha)>\frac{1}{2}|G|$:

\begin{propposition}\label{largeOrdProp}
Let $p$ be a prime, let $G$ be a finite elementary abelian $p$-group of order $p^n$, and let $\alpha\in\Aut(G)$. The following are equivalent:
\begin{enumerate}
\item $\ord(\alpha)>\frac{1}{2}|G|$.
\item One of the following holds:
\begin{itemize}
\item $p>2$, and $\alpha$ is a Singer cycle on $G$, i.e., the primary rational canonical form of $\alpha$ has a single block, of the form $\Comp(P(X))$ where $P(X)\in\IF_p[X]$ is a monic \emph{primitive} irreducible polynomial of degree $n$ (i.e., its roots in $\overline{\IF_p}$ are generators of the cyclic group $\IF_{p^n}^{\ast}$). In this case, $\ord(\alpha)=p^n-1$.
\item $p>2$, $n=2$, and the primary rational canonical form of $\alpha$ has a single block, of the form $\Comp((X-a)^2)$ where $a\in\IF_p^{\ast}$ is of order $p-1$. In this case, $\ord(\alpha)=p(p-1)$.
\item $p=2$, and the blocks of the primary rational canonical form of $\alpha$ are Singer cycles in pairwise coprime dimensions $d_1,d_2,\ldots,d_s>1$ such that $\prod_{i=1}^s{(1-\frac{1}{2^{d_i}})}>\frac{1}{2}$. In this case, $\ord(\alpha)=\prod_{i=1}^s{(2^{d_i}-1)}$.
\end{itemize}
\end{enumerate}
\end{propposition}

\begin{proof}
This follows from the known classification of the pairs $(H,\alpha)$ where $H$ is a finite group and $\alpha$ is an automorphism of $H$ with $\ord(\alpha)>\frac{1}{2}|H|$. Note that each such group $H$ is abelian by \cite[Theorem 1.1.1(1)]{Bor17a}, and thus by \cite[Corollary 1]{Hor74a}, these are just the pairs $(H,\alpha)$ where $\alpha$ has a cycle of length greater than $\frac{1}{2}|H|$, for the classification of which see \cite[Corollary 1.1.8]{Bor16a}.
\end{proof}

\section{Proof of Theorem \ref{mainTheo}(2)}\label{sec3}

The meat of the proof of Theorem \ref{mainTheo}(2) lies in proving the following proposition:

\begin{proposition}\label{mainProp}
Let $G$ be a finite solvable group with $\ffrak(G)>\frac{1}{2}$. Then the derived length of $G$ is at most $3$.
\end{proposition}

We start our observations for the proof of Proposition \ref{mainProp} with the following extension of Proposition \ref{largeOrdProp}:

\begin{proposition}\label{elAbFProp}
Let $p$ be a prime, let $G$ be a finite elementary abelian $p$-group with $|G|=p^n$ and $n>1$, and let $\alpha\in\Aut(G)$. The following are equivalent:
\begin{enumerate}
\item $\ffrak(\alpha)>\frac{1}{2}$.
\item One of the following holds:
\begin{itemize}
\item $\ord(\alpha)>\frac{1}{2}|G|$.
\item $p>2$, and the primary rational canonical form of $\alpha$ has precisely two blocks, one of the form $\Comp(X-1)$ and the other a Singer cycle (see Proposition \ref{largeOrdProp}) in dimension $n-1$.
\item $p=2$, and the primary rational canonical form of $\alpha$ has precisely one block of one of the two forms $\Comp(X-1)$ or $\Comp((X-1)^2)$, and the other blocks of $\alpha$ are Singer cycles in pairwise coprime dimensions $d_1,\ldots,d_s>1$ such that $\prod_{i=1}^s{(1-\frac{1}{2^{d_i}})}>\frac{1}{2}$.
\end{itemize}
\end{enumerate}
In particular, if $\ffrak(\alpha)>\frac{1}{2}$, then $\alpha$ has at most $p$ fixed points in $G$, the exponent of the centralizer of $\alpha$ in $\Aut(G)$ is $\ord(\alpha)$, and $\alpha$ does not admit a proper root in $\Aut(G)$.
\end{proposition}

\begin{proof}
Throughout this proof, we write $G$ as an additive group. We start with the proof of the main statement. For the implication \enquote{(2)$\Rightarrow$(1)}, consider the three listed cases separately:
\begin{itemize}
\item If $\ord(\alpha)>\frac{1}{2}|G|$, then by the definition of $\Ffrak(\alpha)$ from Notation \ref{fNot}(1), and by Lemma \ref{shiftLem}(1), we have $\Ffrak(\alpha)\geq\ord(\alpha)>\frac{1}{2}|G|$.
\item If $p>2$ and the blocks of the primary rational canonical form of $\alpha$ are $\Comp(X-1)$ and a Singer cycle in dimension $n-1$, then let $g$ be a nontrivial element of the $1$-eigenspace of $\alpha$. Since $g$ is a fixed point of $\alpha$, we have by Lemma \ref{shiftLem}(1)
\begin{align*}
\Ffrak(\alpha) &\geq\ord(\A_{\alpha,g})=\ord(\alpha)\cdot\ord(\sigma_{\alpha}(g))=\ord(\alpha)\cdot\ord(\ord(\alpha)g) \\
&=(p^{n-1}-1)\cdot\ord((p^{n-1}-1)g)=(p^{n-1}-1)\cdot p=(1-\frac{1}{p^{n-1}})\cdot|G|>\frac{1}{2}|G|.
\end{align*}
\item In the last case, an analogous argument works if the unique non-Singer-cycle block of $\alpha$ is $\Comp(X-1)$, noting that the orders of the Singer cycle blocks $2^{d_i}-1$ are pairwise coprime by assumption. Hence, assume that we have a direct decomposition $G=H_1\oplus H_2$ into $\alpha$-invariant subspaces with $|H_1|=p^2$, $|H_2|=p^{n-2}$ and $\alpha_{\mid H_1}=:\beta$ is similar to $\Comp((X-1)^2)$ whereas $\alpha_{\mid H_2}=:\gamma$ is similar to a block diagonal matrix whose blocks are the Singer cycle blocks of $\alpha$. Let $g\in H_1$ be a non-fixed point of $\alpha$. Then by Lemma \ref{shiftLem}(1)
\begin{align*}
\Ffrak(\alpha) &\geq\ord(\A_{\alpha,g})=\ord(\alpha)\cdot\ord(\sigma_{\alpha}(g))=2\ord(\gamma)\cdot\ord(\sigma_{\alpha}(g)) \\
&=2\ord(\gamma)\cdot\ord(\ord(\gamma)(g^{\beta}+g))=4\ord(\gamma)=\prod_{i=1}^s{(1-\frac{1}{2^{d_i}})}\cdot|G|>\frac{1}{2}|G|.
\end{align*}
\end{itemize}
We now turn to the proof of \enquote{(1)$\Rightarrow$(2)}.  Let $\alpha\in\Aut(G)$ with $\ffrak(\alpha)>\frac{1}{2}$. If $\alpha$ is fixed-point-free, then by Lemma \ref{shiftLem}(2), we have $\frac{1}{2}|G|<\Ffrak(\alpha)=\ord(\alpha)$, and we are done. Assume thus that $\alpha$ has a nontrivial fixed point. Consider the primary rational canonical form of $\alpha$, with blocks $\Comp(P_i(X)^{e_i})$ for $i=1,\ldots,s$, listed with multiplicities. Since $\alpha$ has a nontrivial fixed point, we know that $P_i(X)=X-1$ for at least one $i\in\{1,\ldots,s\}$. We claim that there is exactly one $i\in\{1,\ldots,s\}$ such that $P_i(X)=X-1$. Assume otherwise. Note that the order of a finite group automorphism cannot exceed the group order (see Lemma \ref{fLem}(1) or \cite[Theorem 2]{Hor74a}) and that by Proposition \ref{frobOrdProp}(1), we have $\ord(\Comp((X-1)^e))=p^{\lceil\log_p(e)\rceil}\leq p^{e-1}$ for all $e\in\IN^+$. This implies that
\[
\ord(\alpha)=\lcm_{i=1,\ldots,s}{\ord(\Comp(P_i(X)^{e_i}))}\leq\prod_{i=1}^s{\ord(\Comp(P_i(X)^{e_i}))}\leq\frac{|G|}{p^2},
\]
and thus by Lemma \ref{shiftLem}(1)
\[
\Ffrak(\alpha)\leq p\cdot\ord(\alpha)\leq\frac{1}{p}|G|\leq\frac{1}{2}|G|,
\]
a contradiction. Henceforth, assume without loss of generality that $P_1(X)=X-1$ and $P_i(X)\not=X-1$ for $i=2,\ldots,s$. Note that a similar argument shows that $e_1\leq 3$, and that $e_1\leq 2$ if $p>2$, because $\lceil\log_p(e)\rceil\leq e-2$ for $e\geq4$ in general as well as for $e\geq3$ if $p>2$. We will now show that in fact, $e_1=1$ if $p>2$, and $e_1\leq 2$ if $p=2$.
\begin{itemize}
\item Assume that $p>2$. Then by the above, we have $e_1\in\{1,2\}$. Assume that $e_1=2$. Setting $\beta:=\Comp((X-1)^2)\in\GL_2(p)$ and applying Lemma \ref{fLem}(2), we find that $\ffrak(\beta)>\frac{1}{2}$. However, we will now show that all shifts of $\beta$ are trivial, which is contradictory, since it implies by Lemma \ref{shiftLem}(1) that
\[
\Ffrak(\beta)=\ord(\beta)=p=\frac{1}{p}|\IF_p^2|<\frac{1}{2}|\IF_p^2|.
\]
To see that the shifts of $\beta$ are trivial, note that under identifying $\IF_p^2$ with $\IF_p[X]/((X-1)^2)$, we can write the shift function $\sigma_{\beta}$ in the form
\begin{align*}
\sigma_{\beta}\colon &\IF_p[X]/((X-1)^2)\rightarrow\IF_p[X]/((X-1)^2), \\
&Q(X)+((X-1)^2)\mapsto (1+X+X^2+\cdots+X^{p-1})Q(X)+((X-1)^2).
\end{align*}
An easy induction on $k\in\IN$ shows that $X^k\equiv kX-(k-1)\Mod{(X-1)^2}$, and thus
\begin{align*}
1+X+X^2+\cdots+X^{p-1} &\equiv\sum_{k=0}^{p-1}{\left(kX-(k-1)\right)}=\frac{1}{2}p(p-1)X-\frac{1}{2}p(p-3) \\
&=0\Mod{(X-1)^2},
\end{align*}
so that all values of $\sigma_{\beta}$ are zero, as asserted.
\item Assume that $p=2$. Then $e_1\in\{1,2,3\}$, and we need to refute the possibility that $e_1=3$. And indeed, if $e_1=3$, then setting $\beta:=\Comp((X-1)^3)\in\GL_3(2)$ and applying Lemma \ref{fLem}(2), we find that $\ffrak(\beta)>\frac{1}{2}$. However, $\ord(\beta)=4$ by Proposition \ref{frobOrdProp}, and so by an argument analogous to the one in the previous bullet point, all shifts of $\beta$ are trivial, since $1+X+X^2+X^3=(X-1)^3\equiv0\Mod{(X-1)^3}$. Hence $\Ffrak(\beta)=\ord(\beta)=4=\frac{1}{2}|\IF_2^3|$, a contradiction.
\end{itemize}
Denote by $\gamma$ the block diagonal matrix with blocks $\Comp(P_i(X)^{e_i})$ for $i=2,\ldots,s$. Applying Lemma \ref{fLem}(2), we find that $\ffrak(\gamma)>\frac{1}{2}$, and using that $\gamma$ is fixed-point-free, we have $\ord(\gamma)=\Ffrak(\gamma)$ by Lemma \ref{shiftLem}(2). Using Proposition \ref{largeOrdProp}, we see that $\alpha$ is of one of the listed types except for one more possibility: $p>2$, $n=3$ and the blocks of the primary rational canonical form of $\alpha$ are $\Comp(X-1)$ and $\Comp((X-a)^2)$ where $\IF_p^{\ast}=\langle a\rangle$.

In order to refute this possibility, note that $\ord(\alpha)=\ord(\Comp((X-a)^2))=p(p-1)$ by Proposition \ref{frobOrdProp}, and write $G=H_1\oplus H_2$ with $H_1$ the ($1$-dimensional) $1$-eigenspace of $\alpha$ and $H_2$ a $2$-dimensional $\alpha$-invariant subspace of $G$ on which $\alpha$ acts as $\Comp((X-a)^2)$. Set $\beta:=\alpha_{\mid H_2}$. We can write an arbitrary element $g\in G$ as $g=g_1+g_2$ with $g_i\in H_i$ for $i=1,2$, and in view of Lemma \ref{shiftLem}(2), we have
\[
\sigma_{\alpha}(g)=\sigma_{\alpha}(g_1)+\sigma_{\alpha}(g_2)=\ord(\alpha)g_1+\sigma_{\beta}(g_2)=0+0=0,
\]
whence all shifts of $\alpha$ are trivial and
\[
\Ffrak(\alpha)=\ord(\alpha)=(1-\frac{1}{p})\frac{1}{p}|G|<\frac{1}{2}|G|,
\]
a contradiction, concluding the proof of the implication \enquote{(1)$\Rightarrow$(2)}.

It remains to prove the \enquote{In particular}. Let $\alpha$ be of one of the three types listed in statement (2). The subgroup $\fix(\alpha)\leq G$ consisting of the fixed points of $\alpha$ is isomorphic to the direct product $\prod_{i=1}^s{\fix(\Comp(P_i(X)^{e_i}))}$. Since
\[
|\fix(\Comp(X-1))|=|\fix(\Comp((X-1)^2))|=p
\]
and all other possible primary Frobenius blocks of $\alpha$ are fixed-point-free, the assertion on the number of fixed points of $\alpha$ follows. As for the assertion on the exponent of $\C_{\Aut(G)}(\alpha)$, proceed as follows: Using Proposition \ref{largeOrdProp} to cover the case \enquote{$\ord(\alpha)>\frac{1}{2}|G|$}, we see that the primary Frobenius blocks of $\alpha$, $\Comp(P_i(X)^{e_i})$ for $i=1,\ldots,s$, have pairwise distinct underlying polynomials $P_i(X)$. By Lemma \ref{centLem}, it follows that
\begin{equation}\label{isomorphismEq}
\C_{\Aut(G)}(\alpha)\cong\prod_{i=1}^s{(\IF_p[X]/(P_i(X)^{e_i}))^{\ast}}.
\end{equation}
We now claim that unless $P_i(X)^{e_i}=X-1$, we have
\[
\Exp((\IF_p[X]/(P_i(X)^{e_i}))^{\ast})=\ord(\Comp(P_i(X)^{e_i})).
\]
Indeed, other than $X-1$, there are only the following two possibilities for $P_i(X)^{e_i}$:
\begin{itemize}
\item $e_i=1$ and $P_i(X)$ is primitive monic irreducible, i.e., $\Comp(P_i(X)^{e_i})$ is a Singer cycle. Then the quotient ring $\IF_p[X]/(P_i(X)^{e_i})$ is isomorphic to the finite field $\IF_{p^{\deg{P_i(X)}}}$, and so
\[
\Exp((\IF_p[X]/(P_i(X)^{e_i}))^{\ast})=p^{\deg{P_i(X)}}-1=\ord(\Comp(P_i(X)^{e_i})).
\]
\item $e_i=2$ and $P_i(X)=X-a$ where $\IF_p^{\ast}=\langle a\rangle$. Let $Q(X)\in\IF_p[X]$ such that $X-a\nmid Q(X)$. Then since $\IF_p[X]/((X-a))$ is isomorphic to $\IF_p$, we have $Q(X)^{p-1}\equiv 1\Mod{X-a}$, i.e.,
\[
Q(X)^{p-1}=1+R(X)(X-a)
\]
for some $R(X)\in\IF_p[X]$. Consequently,
\[
Q(X)^{p(p-1)}=(1+R(X)(X-a))^p=1+R(X)^p(X-a)^p\equiv 1\Mod{(X-a)^2},
\]
and thus
\[
\Exp((\IF_p[X]/((X-a)^2))^{\ast})=p(p-1)=\ord(\Comp((X-a)^2)).
\]
\end{itemize}
In view of Formula (\ref{isomorphismEq}), this implies that $\Exp(\C_{\Aut(G)}(\alpha))=\ord(\alpha)$ except possibly if $\alpha$ has a primary Frobenius block of the form $\Comp(X-1)$. However, in that case, $\alpha$ also has a primary Frobenius block that is a Singer cycle. Since the order of every Singer cycle is divisible by $p-1=\Exp(\IF_p^{\ast})=\Exp((\IF_p[X]/(X-1))^{\ast})$, the equality $\Exp(\C_{\Aut(G)}(\alpha))=\ord(\alpha)$ holds in that case too.

Finally, now that we know that $\Exp(\C_{\Aut(G)}(\alpha))=\ord(\alpha)$, note that if $\alpha$ had a proper root $\beta$ in $\Aut(G)$, then $\beta$ would centralize $\alpha$ and be of strictly larger order than $\alpha$, a contradiction.
\end{proof}

The following is a consequence of Proposition \ref{elAbFProp} and provides an important structural restriction on finite metabelian groups with large $\ffrak$-value:

\begin{proposition}\label{fCyclicProp}
Let $G$ be a finite metabelian group with $\ffrak(G)>\frac{1}{2}$. Then $G'$ is cyclic.
\end{proposition}

\begin{proof}
Assume otherwise. Let $p$ be a prime divisor of $|G'|$ such that the Sylow $p$-subgroup of $G'$ is not cyclic. Through replacing $G$ by $G/\{g^p\mid g\in G'\}$ and using Lemma \ref{fLem}(3), we may assume without loss of generality that $G'$ is an elementary abelian $p$-group with $|G'|\geq p^2$. Let $\alpha\in\Aut(G)$ with $\ffrak(\alpha)>\frac{1}{2}$, and denote by $\tilde{\alpha}$ the automorphism of $G/G'$ induced by $\alpha$. In view of Lemma \ref{fLem}(2), applied with $N:=G'$, and the \enquote{In particular} in Proposition \ref{elAbFProp}, we have that $\Ffrak(\tilde{\alpha})$ is coprime to $\ord(\alpha_{\mid G'})$ -- otherwise, $(\alpha_{\mid G'})^{\Ffrak(\tilde{\alpha})}$ has $\ffrak$-value larger than $\frac{1}{2}$ and admits a proper root in $\Aut(G')$, which is impossible. Set $\beta:=\alpha^{\Ffrak(\tilde{\alpha})}$. Then $\beta$ induces the identity on $G/G'$, but $\beta_{\mid G'}$ generates the same cyclic subgroup of $\Aut(G')$ as $\alpha_{\mid G'}$. Using Lemma \ref{abKerLem}, it follows that $\beta_{\mid G'}$, and thus $\alpha_{\mid G'}$, centralizes the entire image of the conjugation action $G\rightarrow\Aut(G')$ (because elements of $G$ in the same coset of $G'$ correspond to the same conjugation automorphism on $G'$). Let $K$ be the kernel of that conjugation action. Then $K$ is $\alpha$-invariant by Lemma \ref{abKerLem}, and denoting by $\alpha'$ the automorphism of $G/K$ induced by $\alpha$, Lemma \ref{abKerLem} implies that $\alpha'=\id_{G/K}$. However, by Lemma \ref{fLem}(2), we must have $\ffrak(\alpha')>\frac{1}{2}$, whence $G/K$ is cyclic, and the divisibility in Lemma \ref{fLem}(2) implies that $\Ffrak(\alpha')=|G/K|\mid\Ffrak(\tilde{\alpha})$. But $G/K$ is isomorphic to a subgroup of $\C_{\Aut(G')}(\alpha_{\mid G'})=\C_{\Aut(G')}(\beta_{\mid G'})$, which (since $\ffrak(\beta_{\mid G'})>\frac{1}{2}$ by Lemma \ref{fLem}(2)) has exponent $\ord(\beta_{\mid G'})=\ord(\alpha_{\mid G'})$ by the \enquote{In particular} of Proposition \ref{elAbFProp}. Hence if $\ell$ is a prime divisor of $|G/K|$, it follows that
\[
\ell\mid\gcd(\ord(\alpha_{\mid G'}),\Ffrak(\tilde{\alpha}))=1,
\]
a contradiction. Therefore, $K=G$, or equivalently, $G'$ is central in $G$. Hence $G$ is nilpotent of class $2$, and since $G'$ is a $p$-group, all Sylow subgroups of $G$ except the Sylow $p$-subgroup are abelian. Through quotienting out the Hall $p'$-subgroup of $G$ and applying Lemma \ref{fLem}(3), we may assume that $G$ is a $p$-group. Consider the Frattini quotient $G/\Phi(G)$, which is a non-cyclic elementary abelian $p$-group by the Burnside Basis Theorem \cite[result 5.3.2, p.~140]{Rob96a}. Denote by $\alpha_0$ the automorphism of $G/\Phi(G)$ induced by $\alpha$. Then by Lemma \ref{fLem}(2), we have $\ffrak(\alpha_0)>\frac{1}{2}$ and $\Ffrak(\alpha_0)\mid\Ffrak(\tilde{\alpha})$, so that $\gcd(\Ffrak(\alpha_0),\ord(\alpha_{\mid G'}))=1$. Note that $G'\leq\Phi(G)\leq\zeta G$, since $G/\Phi(G)$ is abelian and $G/\zeta G$ is an elementary abelian $p$-group; indeed, for all $x,y\in G$, using \cite[Exercise 5.1.4, p.~128]{Rob96a} (a consequence of the first commutator identity in \cite[result 5.1.5(ii), p.~123]{Rob96a}) and that $G'$ is of exponent $p$ we have
\[
[x^p,y]=[x,y]^p=1_G.
\]
Therefore, the commutator map $[,]:G^2\rightarrow G'$ induces an alternating, skew-symmetric $\IF_p$-bilinear form $[,]:(G/\Phi(G))^2\rightarrow G'$, whose image generates $G'$. Observe that $\ffrak(\alpha_{\mid G'})>\frac{1}{2}$ -- this holds because $\ffrak((\alpha_{\mid G'})^{\Ffrak(\tilde{\alpha})})>\frac{1}{2}$ by Lemma \ref{fLem}(2), and because, as explained above, $\Ffrak(\tilde{\alpha})$ is coprime to $\ord(\alpha_{\mid G'})$, whence one of the conditions from Proposition \ref{elAbFProp} applies to $\alpha_{\mid G'}$ just as it does to $\alpha_{\mid G'}^{\Ffrak(\tilde{\alpha})}$. Now, since $|G'|>p$ and the fixed point subgroup of $\alpha_{\mid G'}$ has at most $p$ elements by the \enquote{In particular} of Proposition \ref{elAbFProp}, we can choose an element $c$ in the image of the said bilinear form that is not fixed under $\alpha$; we denote the cycle length of $c$ under $\alpha$ by $\ell_c$. Let $x\Phi(G),y\Phi(G)\in G/\Phi(G)$ be such that $[x\Phi(G),y\Phi(G)]=c$. Since $c\not=0_G$, we find that $x\Phi(G)$ and $y\Phi(G)$ must be $\IF_p$-linearly independent elements of $G/\Phi(G)$. Since the fixed point subgroup of $\alpha_0$ in $G/\Phi(G)$ has order at most $p$, it follows that at least one of $x\Phi(G)$ or $y\Phi(G)$ is not fixed under $\alpha_0$, and so if $\ell_z$ for $z\in\{x,y\}$ denotes the cycle length of $z\Phi(G)$ under $\alpha_0$, we find that $\lcm(\ell_x,\ell_y)>1$. Since
\[
c^{\left(\alpha^n\right)}=[x\Phi(G),y\Phi(G)]^{\left(\alpha^n\right)}=[(x\Phi(G))^{\left(\alpha_0^n\right)},(y\Phi(G))^{\left(\alpha_0^n\right)}]
\]
for all $n\in\IZ$, it follows that
\[
1<\lcm(\ell_x,\ell_y)\mid\gcd(\ell_c,\ord(\alpha_0))\mid\gcd(\ord(\alpha_{\mid G'}),\Ffrak(\alpha_0))=1,
\]
a contradiction.
\end{proof}

The second statement in the following proposition provides a structural restriction on triples of consecutive (nontrivial) factors in derived series of finite groups: They cannot all be cyclic. This restriction will be crucial for refuting the possibility that a finite solvable group $G$ of derived length at least $4$ can satisfy $\ffrak(G)>\frac{1}{2}$.

\begin{proposition}\label{integralProp}
Let $G$ be a finite group.
\begin{enumerate}
\item Assume that $G$ is metabelian and that both $G'$ and $G/G'$ are cyclic. Then $G'$ admits a semidirect complement $\langle g\rangle$ in $G$, and the conjugation by $g$ on $G'$ is fixed-point-free.
\item Assume that $G$ is solvable of derived length $3$. Then at least one of the three factors in the derived series of $G$ (i.e., $G''$, $G'/G''$ or $G/G'$) is \emph{not} cyclic.
\end{enumerate}
\end{proposition}

\begin{proof}
For statement (1): Let $G'=\langle h\rangle$, and fix an element $g\in G$ mapping onto a generator of $G/G'$ under the canonical projection $G\rightarrow G/G'$. Then $G=\langle g,h\rangle$, whence by \cite[result 5.1.7, p.~124]{Rob96a}, $G'$ is the smallest normal subgroup of $G$ containing $[g,h]$. Using the cyclicity of $G'$, this implies that $G'=\langle [g,h]\rangle$. Write $h^g=h^{\lambda}$ with $\lambda\in\IZ$ and $\gcd(\lambda,|G'|)=1$. Then $[g,h]=h^gh^{-1}=h^{\lambda-1}$, so that $\gcd(\lambda-1,|G'|)=1$ and the conjugation by $g$ on $G'$ is fixed-point-free. In particular, noting that $g$ centralizes each power of itself, we must have $G'\cap\langle g\rangle=\{1_G\}$, and thus $G=\langle g\rangle\ltimes G'$, as required.

For statement (2): Assume otherwise. By statement (1), we have $G'=\langle g\rangle\ltimes G''$ for some $g\in G'$. Since the conjugation by $g$ on $G''$ is fixed-point-free, the induced conjugation action $G'/G''\rightarrow\Aut(G'')$ is nontrivial, so there is a prime divisor $p$ of the index $|G':G''|$ such that the Sylow $p$-subgroup of $G'/G''$ is not contained in the kernel of this action. Let $h$ be a generator of the Sylow $p$-subgroup of $\langle g\rangle$; then $h$ acts nontrivially on $G''$ by conjugation. Set $Q:=G/G''$, and denote by $\pi$ the canonical projection $G\rightarrow Q$. Then, again by statement (1), we have $Q=\langle k^{\pi}\rangle\ltimes\langle g^{\pi}\rangle$ for some $k\in G$, and the conjugation by $k^{\pi}$ on $\langle g^{\pi}\rangle$ is fixed-point-free. This means that $(g^{\pi})^{k^{\pi}}=(g^{\pi})^{\lambda}$ for some $\lambda\in\IZ$ with $\lambda\not\equiv0,1\Mod{\ell}$ for every prime divisor $\ell$ of $\ord(g^{\pi})=\ord(g)$. In particular, $\lambda\not\equiv0,1\Mod{p}$. It follows that $h^k=xh^{\lambda}$ for some $x\in G''$. Since $\langle h\rangle$ is not fully contained in the kernel of the conjugation action $G'\rightarrow\Aut(G'')$, it follows that the conjugations by $h$ and $h^k$ on $G''$ are distinct automorphisms of $G''$, while also being $\Aut(G'')$-conjugate to each other by Lemma \ref{abKerLem}. However, this is impossible, since $\Aut(G'')$ is abelian.
\end{proof}

We are now ready to prove Proposition \ref{mainProp}.

\begin{proof}[Proof of Proposition \ref{mainProp}]
Assume, aiming for a contradiction, that $G$ is a finite solvable group of derived length $\ell\geq4$ such that $\ffrak(G)>\frac{1}{2}$. By Lemma \ref{fLem}(3), we have $\ffrak(G/G^{(4)})>\frac{1}{2}$, so that we may assume that the derived length of $G$ is exactly $4$. Lemma \ref{fLem}(3) lets us infer that $\ffrak(G''),\ffrak(G'/G'''),\ffrak(G/G'')>\frac{1}{2}$, and thus Proposition \ref{fCyclicProp} implies that each of the groups $G'''$, $G''/G'''$ and $G'/G''$ is cyclic. However, these groups are the factors in the derived series of $G'$, yielding a contradiction to Proposition \ref{integralProp}(2).
\end{proof}

Now that the validity of Proposition \ref{mainProp} has been established, we can prove Theorem \ref{mainTheo}(2).

\begin{proof}[Proof of Theorem \ref{mainTheo}(2)]
Let $G$ be a finite group admitting an affine map of order at least $\rho|G|$. Then in particular, $\ffrak(\Rad(G))\geq\ffrak(G)\geq\rho$. Denote by $\ell$ the derived length of $\Rad(G)$, and consider the characteristic series $\Rad(G)=R_0>R_1>\cdots>R_{\lfloor\ell/3\rfloor}\geq\{1_G\}$ in $\Rad(G)$ with $R_i:=\Rad(G)^{(4i)}$ for $i=0,1,\ldots,\lfloor\frac{\ell}{4}\rfloor$. This series has $\lfloor\frac{\ell}{4}\rfloor$ factors that are solvable of derived length $4$, and Lemma \ref{fLem}(3) as well as Proposition \ref{mainProp} let us conclude that
\[
\rho\leq\ffrak(\Rad(G))\leq\prod_{i=0}^{\lfloor\ell/4\rfloor-1}{\ffrak(R_i/R_{i+1})}\leq\prod_{i=0}^{\lfloor\ell/4\rfloor-1}{\frac{1}{2}}=\left(\frac{1}{2}\right)^{\lfloor\ell/4\rfloor},
\]
whence $\log_2(\rho^{-1})\geq\lfloor\frac{\ell}{4}\rfloor$, which is equivalent to $\ell\leq4\cdot\lfloor\log_2(\rho^{-1})\rfloor+3$, the asserted inequality.
\end{proof}

\section{Proof of Theorem \ref{mainTheo}(1)}\label{sec4}

In view of Proposition \ref{mainProp}, Theorem \ref{mainTheo}(1) follows easily from the following:

\begin{proposition}\label{mainProp2}
Let $G$ be a finite group with $\ffrak(G)>\frac{1}{2}$. Then $G$ is solvable.
\end{proposition}

\begin{proof}
Assume otherwise. Since the class of finite solvable groups is closed under group extensions, the quotient $G/\Rad(G)$ is a nontrivial finite group without nontrivial solvable normal subgroups, a so-called nontrivial finite \emph{semisimple} group. By the general structure theory of finite semisimple groups, found e.g.~in \cite[pp.~89ff.]{Rob96a}, the socle (i.e., the subgroup generated by all the minimal nontrivial normal subgroups) of $G/\Rad(G)$ is a nontrivial direct product of nontrivial direct powers $S^n$ of nonabelian finite simple groups $S$. In particular, $G/\Rad(G)$ has a characteristic subgroup of the form $S^n$, and by two applications of Lemma \ref{fLem}(3), we conclude that $\ffrak(S^n)\geq\ffrak(G/\Rad(G))\geq\ffrak(G)>\frac{1}{2}$. In the rest of this proof, we will show that $\ffrak(S^n)\leq\frac{1}{2}$ in order to get a contradiction; our argument is similar to the one in \cite[proof of Theorem 5.2.5]{Bor17a}, where it was shown that $\ffrak(S^n)\leq 1$.

First, assume that $n\geq 3$. Following \cite[proof of Theorem 5.2.5]{Bor17a}, we find that
\[
\Ffrak(S^n)\leq|S^n|^{0.938}=\frac{1}{|S^n|^{0.062}}\cdot|S^n|\leq\frac{1}{60^{3\cdot0.062}}\cdot|S^n|<\frac{1}{2}|S^n|,
\]
a contradiction. Similarly, if $n=2$ and $|S|\geq300$, then
\[
\Ffrak(S^n)\leq|S^n|^{0.938}=\frac{1}{|S^n|^{0.062}}\cdot|S^n|\leq\frac{1}{300^{2\cdot0.062}}\cdot|S^n|<\frac{1}{2}|S^n|.
\]
We now know that if $n>1$, then $n=2$ and $S$ is isomorphic to one of $\Alt(5)$ or $\PSL_2(7)$. Let $\alpha\in\Aut(S^2)=\Aut(S)\wr\Sym(2)$, see \cite[result 3.3.20, p.~90]{Rob96a}. We write $\alpha=\sigma(\alpha_1,\alpha_2)$ with $\alpha_i\in\Aut(S)$ for $i=1,2$ and $\sigma\in\Sym(2)$. If $\sigma$ is trivial, then we have $\ord(\alpha)=\lcm(\ord(\alpha_1),\ord(\alpha_2))$, and otherwise, $2\mid\ord(\alpha)$ and $\alpha^2=(\alpha_2\alpha_1,\alpha_1\alpha_2)$ has conjugate entries, so that $\ord(\alpha)=2\ord(\alpha_1\alpha_2)$. Hence, denoting by $\Ord(H)$ the set of element orders of the finite group $H$ and setting $\mao(H):=\max(\Ord(\Aut(H)))$ (the maximum automorphism order of $H$), we have
\[
\mao(S^2)=\max\{2\mao(S),\max\{\lcm(o_1,o_2)\mid o_1,o_2\in\Ord(\Aut(S))\}\}.
\]
Now consider the two possibilities for $S$ in case $n=2$:
\begin{itemize}
\item If $S=\Alt(5)$, we have $\Ord(\Aut(S))=\Ord(\Sym(5))=\{1,2,3,4,5,6\}$, whence $\mao(S^2)=30$ and, using Lemma \ref{shiftLem}(1),
\[
\Ffrak(S^2)\leq\mao(S^2)\cdot\Exp(S^2)=\mao(S^2)\cdot\Exp(S)=30\cdot30=900=\frac{1}{4}|S^2|,
\]
a contradiction.
\item If $S=\PSL_2(7)$, we have $\Ord(\Aut(S))=\Ord(\PGL_2(7))=\{1,2,3,4,6,7,8\}$, whence $\mao(S^2)=56$ and
\[
\Ffrak(S^2)\leq\mao(S^2)\cdot\Exp(S^2)=\mao(S^2)\cdot\Exp(S)=56\cdot84=4704=\frac{1}{6}|S^2|,
\]
another contradiction.
\end{itemize}
We thus conclude that $n=1$. For $\alpha\in\Aut(S)$, set $\overline{\alpha}:=\conj_{\Aut(S)}(\alpha)$, a natural extension of $\alpha$ on $S$ under the identification of $S$ with $\Inn(S)$. For $x\in S$, the affine map $\overline{A}:=\A_{\overline{\alpha},x}$ of $\Aut(S)$ is an extension of the affine map $A:=\A_{\alpha,x}$ of $S$, whence $\ord(A)\mid\ord(\overline{A})$. Since $\Aut(S)$ is complete, it follows by Lemma \ref{shiftLem}(3) that $\Ffrak(\alpha)\mid\Ffrak(\overline{\alpha})=\Exp(\Aut(S))$, whence
\begin{equation}\label{FfrakEq}
\Ffrak(S)\leq\Exp(\Aut(S)).
\end{equation}
We now make a case distinction according to the classification of finite simple groups:
\begin{enumerate}
\item Case: $S$ is sporadic (including the Tits group). Then by Formula (\ref{FfrakEq}), we have $\Ffrak(S)\leq\Exp(\Aut(S))\leq|\Out(S)|\cdot\Exp(S)<\frac{1}{2}|S|$, where the last inequality can be verified by reading off $|\Out(S)|$ and $\Exp(S)$ from the information on sporadic groups available from the ATLAS of Finite Groups \cite{ATLAS}.
\item Case: $S$ is alternating. We distinguish a few subcases:
\begin{enumerate}
\item First, assume that $S=\Alt(5)$. Let $\alpha\in\Aut(S)$. If $\alpha\in\Inn(S)$, then by Lemma \ref{shiftLem}(3), $\Ffrak(\alpha)=\Exp(S)=30=\frac{1}{2}|S|$. And if $\alpha\notin\Inn(S)$, then by Lemma \ref{shiftLem}(3), applied to $\Aut(S)\cong\Sym(5)$, we have for each $x\in S$ that
\begin{align*}
\ord(\A_{\alpha,x}) &=\lcm(\ord(\alpha),\ord(x\alpha))\mid\lcm(\{\ord(\chi)\mid \chi\in\Sym(5)\setminus\Alt(5)\}) \\
&=\lcm(2,4,6)=12,
\end{align*}
whence $\Ffrak(\alpha)\leq12<30=\frac{1}{2}|S|$.
\item Next, assume that $S=\Alt(n)$ with $n\in\{6,7\}$. Then by Formula (\ref{FfrakEq}), we have $\Ffrak(S)\leq\Exp(\Aut(S))<\frac{1}{2}|S|$, since $\Exp(\Aut(\Alt(6)))=120$, read off from the ATLAS of Finite Groups \cite[p.~5]{ATLAS}, and $\Exp(\Aut(\Alt(7)))=\Exp(\Sym(7))=420$.
\item Finally, assume that $S=\Alt(n)$ with $n\geq8$. Denote by $\psi:\IN^+\rightarrow\IN^+,m\mapsto\log{\Exp(\Sym(m))}$, the second Chebyshev function. Using \cite[Theorem 12]{RS62a} and Formula (\ref{FfrakEq}), we conclude that
\[
\Ffrak(S)\leq\Exp(\Aut(S))=\Exp(\Sym(n))=\e^{\psi(n)}<\e^{1.03883n}<\frac{1}{4}n!=\frac{1}{2}|S|.
\]
\end{enumerate}
\item Case: $S$ is of Lie type, say in defining characteristic $p$. Denote by $\Out(S)$ the outer automorphism group $\Aut(S)/\Inn(S)$ of $S$. Note that by Formula (\ref{FfrakEq}) and the inequalities $\Exp(\Aut(S))\leq\Exp(S)\cdot|\Out(S)|$ and $\Exp(S)_{p'}\leq|S|_{p'}$, it is sufficient to show that
\begin{equation}\label{expspEq}
\Exp(S)_p\cdot|\Out(S)|\leq\frac{1}{2}|S|_p
\end{equation}
in order to obtain $\Ffrak(S)\leq\frac{1}{2}|S|$. If $X_d$ is the root system associated with the simple linear algebraic group $X_d(\overline{\IF_p})$ from which the inner diagonal automorphism group of $S$ can be obtained as the fixed point subgroup of a suitable Frobenius map on $X_d(\overline{\IF_p})$ (as described in \cite[Section 3, Notation]{Har92a}), then by \cite[Corollary 0.5]{Tes95a} and \cite[Theorem, p.~84]{Hum90a}, we have $\Exp(S)_p=p^{\lceil\log_p{h(X_d)}\rceil}$, where $h(X_d)$ is the Coxeter number of the root system $X_d$. Using this as well as the tabulated values of $h(X_d)$ from \cite[Table 2, p.~80]{Hum90a} and the formulas for $|\Out(S)|$ obtained from \cite[Table 5, p.~xvi]{ATLAS}, we find that Formula (\ref{expspEq}) holds except in the following cases:
\begin{itemize}
\item $S=A_1(p)=\PSL_2(p)$.
\item $S$ is one of the finitely many groups $A_d(q)=\PSL_{d+1}(q)$ for
\[
(d,q)\in\{(1,4),(1,8),(1,9),(1,27),(1,25),(1,49),(2,2),(2,4)\}.
\]
\end{itemize}
For dealing with $S=\PSL_2(p)$, we can use the following argument, which generalizes the one for $\Alt(5)$ from above: We know that $S$ is of index $2$ in $\Aut(S)\cong\PGL_2(p)$. Let $\alpha\in\Aut(S)$. If $\alpha\in\Inn(S)$, then by Lemma \ref{shiftLem}(3), we have $\Ffrak(\alpha)=\Exp(S)\mid\frac{1}{2}|S|$, where the divisibility holds because the Sylow $2$-subgroups of $S$ are not cyclic \cite[Lemma 1.4.1, p.~34]{ALSS11a}. If, on the other hand, $\alpha\notin\Inn(S)$, then by Lemma \ref{shiftLem}(3), applied to $\Aut(S)$, we have for each $x\in S$ that
\[
\ord(\A_{\alpha,x})\mid\lcm(\{\ord(\beta)\mid \beta\in\PGL_2(p)\setminus\PSL_2(p)\})\mid\frac{2\Exp(S)}{p}\mid\frac{|S|}{p},
\]
where the second divisibility holds because $\PGL_2(p)\setminus\PSL_2(p)$ contains no elements of order divisible by $p$. In particular, $\Ffrak(\alpha)\leq\frac{|S|}{p}\leq\frac{|S|}{2}$.

Finally, each of the remaining finitely many groups $S$ listed above except $\PSL_2(49)$ has an entry in the ATLAS of Finite Groups \cite{ATLAS}, from which we can read off $\Exp(\Aut(S))$ and check that it is at most $\frac{1}{2}|S|$, so that $\Ffrak(S)\leq\frac{1}{2}|S|$ by Formula (\ref{expspEq}). For $S=\PSL_2(49)$, we use the facts that $\Out(S)\cong(\IZ/2\IZ)^2$ and that the Sylow $2$-subgroups of $S$ are not cyclic \cite[Lemma 1.4.1, p.~34]{ALSS11a} to conclude that
\begin{align*}
\Exp(\Aut(S)) &\leq\Exp(S)\cdot\Exp(\Out(S))=2\Exp(S)=2\cdot7^{\lceil\log_7(2)\rceil}\cdot\Exp(S)_{7'} \\
&=14\Exp(S)_{7'}\leq14\cdot\frac{|S|_{7'}}{2}=7|S|_{7'}=\frac{1}{7}|S|<\frac{1}{2}|S|.
\end{align*}
\end{enumerate}
\end{proof}

\begin{proof}[Proof of Theorem \ref{mainTheo}(1)]
Let $G$ be a finite group with $\ffrak(G)>\frac{1}{2}$. By Proposition \ref{mainProp2}, $G$ is solvable, and by Proposition \ref{mainProp}, the derived length of $G$ is at most $3$, as required.
\end{proof}

\section{Some related open problems and questions}\label{sec5}

A natural question for spurring further research based on known results is whether those results can be improved in some way. With regard to Theorem \ref{mainTheo}(1), we note that the words \enquote{solvable of derived length at most $3$} cannot be replaced by \enquote{abelian}, since finite dihedral groups have affine maps that move all group elements in one cycle, see \cite[Examples 2.3(ii)]{JS75a}. However, the author knows of no examples that achieve derived length $3$, making the following a natural question for further research:

\begin{question}\label{ques1}
Is a finite group $G$ admitting an affine map of order larger than $\frac{1}{2}|G|$ necessarily metabelian?
\end{question}

Note that if one can even show that a finite group $G$ with $\ffrak(G)>\frac{1}{2}$ is metabelian, then this also improves the bound in Theorem \ref{mainTheo}(2) to $3\cdot\lfloor\log_2(\rho^{-1})\rfloor+2$. Moreover, in the hypothetical improved version of Theorem \ref{mainTheo}(1) where \enquote{solvable of derived length at most $3$} would be replaced by \enquote{metabelian}, the constant $\frac{1}{2}$ would be optimal, because by Lemma \ref{shiftLem}(3), the complete group $G=\Sym(4)$, which is solvable of derived length $3$, has an affine map of order $\lcm(3,4)=12=\frac{1}{2}|G|$.

The next problem is more ambitious than Question \ref{ques1}. To motivate it, we note that Theorem \ref{mainTheo}(1) has an analogue for automorphisms, \cite[Theorem 1.1.1(1)]{Bor17a}: A finite group $G$ with an automorphism $\alpha$ such that $\ord(\alpha)>\frac{1}{2}|G|$ is abelian. Based on this, one can classify the pairs $(G,\alpha)$ where $G$ is a finite group and $\alpha$ is an automorphism of $G$ with $\ord(\alpha)>\frac{1}{2}|G|$, see \cite[Corollary 1.1.8]{Bor16a} and \cite[Corollary 1]{Hor74a}. It would be interesting to achieve such a classification for affine maps too:

\begin{problem}\label{prob1}
Classify the pairs $(G,A)$ where $G$ is a finite group and $A$ is an affine map of $G$ with $\ord(A)>\frac{1}{2}|G|$.
\end{problem}

Through Theorem \ref{mainTheo}(2) and \cite[Theorem 1.1.3(2)]{Bor17a}, we now know that both the index and the derived length of the solvable radical of a finite group $G$ with an affine map of order at least $\rho|G|$ are (explicitly) bounded from above in terms of $\rho$. One may ask whether even stronger restrictions on the structure of $G$ can be inferred from this assumption. For example, it follows from a result of Neumann \cite[Theorem 1]{Neu89a} and was also observed later by Guralnick and Robinson \cite[Lemma 2(iii) and Theorems 10(ii) and 12(i)]{GR06a} that both the index of the Fitting subgroup of $G$ (the largest nilpotent normal subgroup of $G$) and the derived length of $\Rad(G)$ can be bounded from above in terms of the \emph{commuting probability of $G$}, i.e., the number
\[
\frac{|\{(x,y)\in G^2\mid xy=yx\}|}{|G|^2}.
\]
This raises the following question:

\begin{question}\label{ques2}
If a finite group $G$ admits an affine map $A$ of order at least $\rho|G|$ for a given $\rho\in\left(0,1\right]$, then is the commuting probability of $G$ bounded from below by an (explicit) positive-valued function in $\rho$? If not, does the answer become \enquote{yes} at least if $A$ is assumed to be an automorphism of $G$?
\end{question}

The last question is not concerned with affine maps of large order, but with the relationship between the order and the largest cycle length of an affine map. It was observed by Horo\v{s}evski\u{\i} that the largest cycle length of an automorphism $\alpha$ of a finite group $G$ can be strictly smaller than $\ord(\alpha)$, see \cite[remarks after Corollary 1]{Hor74a}. However, he showed that these two numbers are equal if $G$ is nilpotent \cite[Corollary 1]{Hor74a}, or if $G$ has no nontrivial solvable normal subgroups \cite[Theorem 1]{Hor74a}. The latter result was generalized by the author to affine maps, see \cite[Theorem 3.5.1]{Bor17a}, but for nilpotent groups, it is still open:

\begin{question}\label{ques3}
Is it true that for every finite nilpotent group $G$ and every affine map $A$ of $G$, the order of $A$ is equal to the largest cycle length of $A$ on $G$?
\end{question}

\end{document}